\documentclass[a4paper,12pt]{amsart}

\usepackage[
hmarginratio={1:1},     % equal left and right margins
vmarginratio={1:1},     % equal top and bottom margins
textwidth=15cm,        % new text width
textheight=21cm,			% new text height
heightrounded,          % always useful
]{geometry}

\usepackage[utf8]{inputenc}
\usepackage[UKenglish]{babel}
\usepackage{csquotes}
\usepackage{graphicx}
\usepackage[pagebackref=false]{hyperref}
\usepackage{amsfonts, amstext, amsmath, amsthm, amscd, amssymb}
\usepackage[sc]{mathpazo}
\usepackage{mathtools}
\usepackage[dvipsnames]{xcolor}
\usepackage{tikz}
\usepackage{tikz-cd}
\usetikzlibrary{positioning}
\usepackage{float}
\usepackage{svg}
\usepackage[
    backend=biber,
    style=alphabetic,
    sorting=nyt,
	maxnames=10,
    url=false,
    isbn=false,
]{biblatex}

\hypersetup{
  colorlinks=true,
  linkcolor=blue,
  citecolor=blue,
  urlcolor=blue,
}

\definecolor{dark-red}{rgb}{0.4,0.15,0.15}
\definecolor{dark-blue}{rgb}{0.15,0.15,0.4}
\definecolor{dark-green}{rgb}{0.15,0.4,0.15}

\numberwithin{equation}{section}

\theoremstyle{plain}
\newtheorem{theorem}{Theorem}[section]

\newtheorem{lemma}[theorem]{Lemma}

\theoremstyle{definition}
\newtheorem{definition}[theorem]{Definition}
\newtheorem{example}[theorem]{Example}
\newtheorem{remark}[theorem]{Remark}

\newcommand{\variable}{}
\newtheorem*{thm_again*}{Theorem~\ref*{\variable}}
\newenvironment{thm_again}[2]
 {
  \renewcommand{\variable}{#1}
  \begin{thm_again*}[#2]
  }
 {
  \end{thm_again*}
  }

\newcommand{\s}{\mathbb{S}}
\newcommand{\ZZ}{\mathbb{Z}}
\newcommand{\RR}{\mathbb{R}}

\newcommand{\tr}{\triangleright}
\newcommand{\tl}{\triangleleft}
\renewcommand{\epsilon}{\varepsilon}
\renewcommand{\phi}{\varphi}

\DeclareMathOperator{\rk}{rk}
\DeclareMathOperator{\cl}{cl}
\newcommand{\retract}{\mathbin{\rotatebox[origin=c]{-45}{$\rightsquigarrow$}}}

\begin{document}

\title{The fundamental quandle of ribbon concordances}
\author{Eva Horvat, Luka Marčič}
\address{University of Ljubljana, Faculty of Education, Kardeljeva plo\v s\v cad 16, 1000 Ljubljana, Slovenia \newline
	University of Ljubljana, Faculty of Mathematics and Physics, Jadranska ulica 19, 1000 Ljubljana, Slovenia}
\email{\href{mailto:eva.horvat@pef.uni-lj.si}{eva.horvat@pef.uni-lj.si}\\
	\href{mailto:luka.marcic@fmf.uni-lj.si}{luka.marcic@fmf.uni-lj.si}}
\date{\today}

\begin{abstract}
  	We describe the fundamental quandle of a properly embedded surface $\Sigma$ (possibly with boundary) in $\RR ^{3}\times I$
	and derive its presentation in terms of its motion picture.
	Our study is based on the topological definition of the fundamental quandle. 
	We prove that a ribbon concordance $C$ from a classical knot $K_1$ to a classical knot $K_0$ 
	gives rise to a quandle monomorphism
	$Q(K_0)\hookrightarrow Q(C)$ and a quandle epimorphism $Q(K_1)\twoheadrightarrow Q(C)$.
\end{abstract}

\maketitle

\section{Introduction}
Quandles are algebraic structures whose axioms are designed to mimic the Reidemeister moves between classical link diagrams. 
Introduced independently by David Joyce in 1982~\cite{J} and Sergei Matveev in 1984~\cite{M}, 
the fundamental quandle provides an almost complete invariant of ambient isotopy of classical knots (it is a complete
invariant for what is known as weak equivalence). 
From the turn of the century onwards, our knowledge about the algebra of quandles and their 
applications in knot theory has progressed rapidly,
resulting in a myriad of powerful computable invariants.

The fundamental quandle of a properly embedded codimension two submanifold of a connected manifold was topologically 
defined in~\cite{J} and~\cite{FR}. 
Quandle and biquandle cocycle invariants of closed embedded surfaces in 4-space have been studied 
by several authors.
These methods are usually based on some diagram-based algebra without a 
clear connection with the topological definition. 
In this article, we attempt to partly rectify this by explaining the relationship between the fundamental quandle of an 
embedded surface $F$ with boundary and the fundamental quandles of its boundary components. 
We construct a presentation of $Q(F)$ from a motion picture diagram of $F$, and then, using
our construction for a ribbon concordance between two classical knots, prove a quandle analogue of Gordon's lemma 
about the knot groups of two ribbon concordant knots:  
\begin{lemma}[\cite{G}, lemma~3.1]
	\label{lemma:gordon}
	If $C \subset \s^3\times I$ is a ribbon concordance from $K_1$ to $K_0$, 
	then the inclusions $i_0: \s^3-K_0 \hookrightarrow \s^3 \times I - C$
	and $i_1:\s^3-K_1 \hookrightarrow \s^3 \times I - C$ induce a monomorphism
	$\iota_0: \pi _{1}(\s^3 - K_0) \hookrightarrow \pi _{1}(\s^3 \times I -C)$ 
	and an epimorphism $\iota_1: \pi _{1}(\s^3 - K_1)\twoheadrightarrow \pi _{1}(\s^3 \times I - C)$,
	respectively.
\end{lemma}
This gives us the main result of this article.

\begin{theorem}
	\label{thm:ribbon_conc_quandle}
	If $C \subset \s^3\times I$ is a ribbon concordance from $K_1$ to $K_0$, 
	then the inclusions of pairs $i_0: (\s^3 \times \{0\}, K_0 \times \{0\}) \hookrightarrow (\s^3 \times I, C)$
	and $i_1:(\s^3 \times \{1\}, K_1 \times \{1\}) \hookrightarrow (\s^3 \times I, C)$ induce a quandle monomorphism
	$\iota_0: Q(\s^3, K_0) \hookrightarrow Q(\s^3 \times I, C)$ 
	and a quandle epimorphism $\iota_1: Q(\s^3, K_1)\twoheadrightarrow Q(\s^3 \times I,C)$,
	respectively.
\end{theorem}

The article is organized as follows. 
In section~\ref{sec2} we review the basics on quandles and their presentations, recall the topological definition of the fundamental quandle 
of a codimension two submanifold and summarize some known results about quandles, then prove a Sei\-fert-Van-Kam\-pen-like theorem we need later.
Section~\ref{sec3} discusses the fundamental quandles of embedded surfaces (possibly with boundary), 
obtaining a presentation of the fundamental quandle 
of a properly embedded surface in $\RR ^{3}\times I$ from its motion picture.
Finally, section~\ref{sec4} deals with the special case of ribbon concordances, proving theorem~\ref{thm:ribbon_conc_quandle}.
We conclude the article with a brief discussion on possible applications and future research directions.

\section*{Acknowledgements}
The authors would like to thank Michał Jabłonowski for helpful and insightful comments pertaining to the
earlier draft of this article, and for pointing out a number of mistakes and inaccuracies.

Both authors are supported by the Slovenian Research Agency grant P1-0292.

\section{Preliminaries on quandles}
\label{sec2}

Throughout the article, we work in the smooth category and all our (sub)\-ma\-ni\-folds are oriented.
We denote by $I$ the standard unit interval $[0,1]$.

\begin{definition}
	A \textit{quandle} is a set $Q$ with a binary operation $\tr \colon Q\times Q\to Q$ that satisfies the following axioms: 
	\begin{enumerate}
		\item $x\tr x=x$,
		\item the map $R_{y}\colon Q\to Q$, given by $R_{y}(x)=x\tr y$, is a bijection, and
		\item $(x\tr y)\tr z=(x\tr z)\tr (y\tr z)$
	\end{enumerate} for every $x,y,z\in Q$. 
\end{definition}
The second axiom is equivalent to an existence of a second operation defined by $x\tl y=R_{y}^{-1}(x)$, 
such that $(x\tr y)\tl y=x=(x\tl y)\tr y$ for every $x,y\in Q$. 
A \textit{quandle homomorphism} is a function between two quandles $f\colon (Q_{1},\tr _{1})\to (Q_{2},\tr _{2})$ 
that satisfies $f(x\tr _{1}y)=f(x)\tr _{2}f(y)$ for every $x,y\in Q_1$.

\begin{example}[Conjugation quandle]
	In any group $G$, the operation of conjugation $a\tr b:=b^{-1}ab$ defines a quandle, denoted by $conj(G)$, with
	$a \tl b = b a b^{-1}$.
\end{example}

Let $Q$ be a set with a binary operation $\tr \colon Q\times Q\to Q$ (a magma). 
An equivalence relation $\sim $ on $Q$ is called a $\tr $-\textit{congruence} if $a\sim b\land c\sim d$ implies $a\tr c\sim b\tr d$ for every $a,b,c,d\in Q$.
Observe that the induced operation $\tr $ on the quotient $Q/_{\sim }$, given by $[a]\tr [b]=[a\tr b]$, is well-defined iff $\sim $ is a $\tr $-congruence. 
If $(Q,\tr )$ is a quandle and $\sim $ is a $\tr $-congruence and a $\tl $-congruence, then $(Q/_{\sim },\tr )$ is also a quandle.
Note that both congruence conditions are necessary for a rigorous definition, as was shown in~\cite{BT}.

Let $S$ be a set and denote by $F(S)$ the free group over $S$. 
Define operations $\tr$ and $\tl$ on $S \times F(S)$ by $(a,w) \tr (b,z) = (a,w\overline{z}bz)$
and $(a,w) \tl (b,z) = (a,w\overline{z}\overline{b}z)$.
Let an equivalence relation $\sim _{Q}$ on the product $S\times F(S)$ be given by 
$$(a,w_1)\sim _{Q}(b,w_2) \Leftrightarrow a=b\land w_{2}=a^{k}w_1\textrm{ for some }k\in \ZZ \;.$$  
We will denote the quotient set $(S\times F(S))/_{\sim _{Q}}$ by $Q_S$. 
Define an inclusion $\iota \colon S\to Q_{S}$ by $\iota (a)=[a, 1]$.

\begin{lemma} 
	The above relation $\sim _{Q}$ is the smallest $\tr $-congruence and $\tl $-congruence on $S\times F(S)$ for which the quotient set 
	$(S\times F(S))/_{\sim _{Q}}$ with induced operation is a quandle.

	Moreover, for any quandle $Q'$ and any map $f: S \to Q'$, there is a unique quandle homomorphism $F: Q_S \to Q'$
	such that $F(\iota(a)) = f(a)$ for all $a \in S$.
\end{lemma}
\begin{proof}
	We first prove that $\sim _Q$ is a $\tr $-congruence and a $\tl $-congruence.
	Suppose that $(a,w_1)\sim _{Q}(a,w_2)$ and $(b,v_1)\sim _{Q}(b,v_2)$ in $S\times F(S)$. 
	It follows that $w_{2}=a^{k}w_{1}$ and $v_{2}=b^{l}v_{1}$ for some $k,l\in \ZZ $, thus $(a,w_i)\tr (b,v_i)=(a,w_{i}\overline{v}_{i}bv_{i})$ and 
	$w_{2}\overline{v}_{2}bv_{2}=a^{k}w_{1}\overline{v}_{1}bv_{1}$. 
	Therefore, $\sim _{Q}$ is a $\tr $-congruence and the operation $\tr $ is well-defined on the equivalence classes.
	Similarly, $\sim _Q$ is a $\tl $-congruence and $\tl$ is well-defined on the equivalence classes.

	For the remainder of the proof, we slightly abuse notation and perform all computations on representatives of the equivalence classes. 

	To verify that $Q_S$ is a quandle, first observe that 
	$$((a,w) \tr (b,z)) \tl (b,z) = (a,w\overline{z}bz) \tl (b,z) = (a, w \overline{z} b z \overline{z} \overline{b} z) = (a,w)$$
	and that, likewise, $((a,w) \tl (b,z)) \tr (b,z) = (a,w)$, showing that the quandle axiom (2) holds.
	Another computation
	\begin{align*}
		\left ((a, w)\tr (c, z)\right )\tr \left ((b, v)\tr (c, z)\right ) &=(a, w\overline{z}cz)\tr (b,v\overline{z}cz) =(a,w\overline{z}cz\overline{z}\overline{c}z\overline{v}bv\overline{z}cz)=\\
		& =(a,w\overline{v}bv\overline{z}cz)=\left ((a,w)\tr (b,v)\right )\tr (c,z)
	\end{align*}
	takes care of the quandle axiom (3). For the quandle axiom (1), compute 
	$$(a,w)\tr (a,w)=(a,w\overline{w}aw)=(a,aw)\sim _{Q}(a,w)$$ 
	and observe that a quotient $(S\times F(S))/_{\sim }$ is a quandle exactly when the congruence $\sim $ contains $\sim _{Q}$.

	Let now $(Q', \tr_{Q'})$ be a quandle and $f: S \to Q'$ an arbitrary map. Notice that we can decompose every element $(a,w) \in Q_S$ as a quandle product
	$$( \dots (((a, 1) \diamond^{\epsilon_1} (b_1, 1)) \diamond^{\epsilon_2} (b_2, 1)) \diamond^{\epsilon_3} \dots ) \diamond^{\epsilon_n} (b_n, 1), $$
	where $b_1, \dots, b_n \in S$, $\epsilon_1, \dots, \epsilon_n \in \{-1,1\}$, $\diamond^{1} = \tr$ and $\diamond^{-1} = \tl$, such that
	$w = b_1^{\epsilon_1} \dots b_n^{\epsilon_n}$.
	We can then recursively define a quandle homomorphism $F: Q_S \to Q'$ by defining
	$F(a,1) = f(a)$ and $F(a, w \cdot b) = F(a,w) \tr_{Q'} f(b)$ for all $a,b \in S$ and $w \in F(S)$. Notice that the map $f$ is well-defined
	on $\sim _Q$-classes.
	Then $F(\iota(a)) = F(a,1) = f(a)$ for
	all $a \in S$ and it follows from the canonical form above that $F$ is the only quandle homomorphism $Q_S \to Q'$ with this property, completing the proof.
\end{proof}

\begin{definition}
	The quandle $Q_{S}=(S\times F(S))/_{\sim _{Q}}$ is called the \textit{free quandle} over the set $S$.
\end{definition}

The importance of the free quandle comes from the fact that it allows us to define presentations of quandles.

\begin{definition}
	Let $Q$ be a quandle. A \textit{presentation} $\langle X\mid R\rangle$ for $Q$, where $X$ is a set and $R \subset Q_X \times Q_X$,
	is an isomorphism between $Q$ and the quotient quandle $Q_X/{\sim_R}$, where $\sim_R$ is the smallest congruence on $Q_X$ containing $R$. 
	To simplify notations, we will often write $x=y$ instead of $(x,y) \in R$ and $\langle x_1,\ldots, x_m  \mid r_1,\ldots, r_n\rangle$  
	instead of $\langle \{x_1,\ldots, x_m\}\mid \{r_1,\ldots, r_n\}\rangle$.
	The cardinality of the smallest subset $S \subset Q$ such that $\langle S \rangle = Q$ is called the \textit{rank} of $Q$ and is denoted by $\rk(Q)$.

	The \textit{free product} of quandles $Q_1$ and $Q_2$ given by presentations
	$$Q_1 = \langle x_1, \dots, x_n \mid r_1, \dots, r_k\rangle \quad \text{and} \quad
	Q_2 = \langle y_1, \dots, y_n \mid s_1, \dots, s_l\rangle$$ is defined as the quandle $Q_1 * Q_2$ with presentation
	$$Q_1 * Q_2 = \langle x_1, \dots, x_n, y_1, \dots, y_n \mid r_1, \dots, r_k, s_1, \dots, s_l\rangle.$$
\end{definition}

Next, let us recall the definition of the fundamental quandle of a codimension two submanifold of a topological manifold,
as given originally by Fenn and Rourke~\cite{FR}.

\begin{definition}
	Let $L\subset M$ be a properly embedded codimension two subma\-ni\-fold (possibly with boundary) of a connected manifold $M$.
	Assume that $L$ is transversely oriented in $M$, denote by $N_{L}$ the normal disk bundle and by $E_{L}=\cl(M-N_{L})$ its exterior. 
	Choose a basepoint $z\in E_L$.
	Define
	$$\Gamma_L := \frac{\mathcal{C}( (I, \{0\}, \{1\}) \to (E_L, \partial N_L, \{z\}))}{\text{homotopy}}\; ,$$
	the space of (continuous) paths from $\partial N_L$ to $z$ modulo homotopies fixing the endpoint $z$ and allowing the initial point to wander freely along $\partial N_L$.
	For any point $p\in \partial N_{L}$, denote by $m_{p}$ the loop in $\partial N_{L}$ based at $p$, which follows around the meridian of $L$ in the positive direction.
	The \textit{fundamental quandle} $Q(M, L)$ of $L$ in $M$ is the set $\Gamma _{L}$ together with operation
	$$[\alpha]\tr [\beta] := [\alpha \cdot \overline{\beta }\cdot m_{\beta (0)}\cdot \beta ] \;,$$
	where $\overline{\beta}$ denotes the reversal of the path $\beta$ and $\cdot $ denotes the concatenation of paths.
\end{definition}

Note that different choices of tubular neighbourhoods and basepoints give rise to isomorphic quandles, and that
the fundamental quandle is an invariant of $L$ in $M$ up to ambient isotopy.

\begin{example}[Fundamental quandle of a link]
	\label{example:classic_quandle}
	A presentation of the fundamental quandle of a link $L$ in $\RR^{3}$ may be obtained from any diagram of $L$; 
	namely, every arc of the diagram corresponds to a generator and every crossing gives a crossing relation 
	(see Figure~\ref{fig:quandle_crossing_relations}).
\end{example}

\begin{figure}[H]
	\centering
	\begin{tikzpicture}
		\draw[black, thick][->] (-1,-1) -- (1,1);
		\draw[black, thick][<-] (-1,1) -- (-0.1,0.1);
		\draw[black, thick] (0.1,-0.1) -- (1,-1);
		\filldraw[black] (-0.6,0.5) node[anchor=east]{$z$};
		\filldraw[black] (-0.6,-0.5) node[anchor=east]{$y$};
		\filldraw[black] (0.6,-0.5) node[anchor=west]{$x$};
  		\filldraw[black] (0,-1.2) node[anchor=north]{$z = x\tr y$};

		\draw[black, thick][<-] (3,1) -- (5,-1);
		\draw[black, thick] (3,-1) -- (3.9,-0.1);
		\draw[black, thick][->] (4.1,0.1) -- (5,1);
		\filldraw[black] (4.6,0.5) node[anchor=west]{$z$};
		\filldraw[black] (4.6,-0.5) node[anchor=west]{$y$};
		\filldraw[black] (3.4,-0.5) node[anchor=east]{$x$};
  		\filldraw[black] (4,-1.2) node[anchor=north]{$z = x\tl y$};
	\end{tikzpicture}
	\caption{The quandle crossing relations}
	\label{fig:quandle_crossing_relations}
\end{figure}

To finish the section, we prove
a Sei\-fert-Van-Kam\-pen-like theorem for fundamental quandles.

\begin{theorem}
	\label{thm:svk}
	Let $M$ be one of $\RR^3, \s^3,  \RR^3 \times I$ or $\s^3 \times I$ and
	let $L\subset M$ be a properly embedded codimension two subma\-ni\-fold (possibly with boundary) of a connected manifold $M$. Assume
	that $L$ is transversely oriented in $M$. Let $\{U, V\}$ be an open cover
	of $M$ such that the spaces $U - L \cap U$, $V-L \cap V$ and $U \cap V - L \cap U \cap V$  all path connected. Finally, assume
	that $U$, $V$ and $U \cap V$ are simply connected. Then the fundamental
	quandle $Q(M, L)$ is given by the pushout diagram
	\begin{figure}[H]
		\centering
		\begin{tikzcd}
			Q(U \cap V, L \cap U \cap V) \ar[r, "\iota_U"] \ar[d, "\iota_V"] 	&	Q(U, L \cap U) \ar[d, dashrightarrow]\\
			Q(V, L \cap V) \ar[r, dashrightarrow]								&	Q(M, L)
		\end{tikzcd}
	\end{figure}
	where $\iota_U$ and $\iota_V$ are the quandle homomorphisms induced by the inclusions 
	$$(U \cap V, L \cap U \cap V) \hookrightarrow (U, L \cap U) \quad \text{and} \quad
	(U \cap V, L \cap U \cap V) \hookrightarrow (V, L \cap V),$$
	respectively. In other words, $Q(M, L)$ is the amalgamated product
	$$Q(M, L) \cong Q(U, L \cap U)*_{Q(U \cap V, L \cap U \cap V)}Q(V, L \cap V).$$
\end{theorem}
\begin{proof}
	Throughout the proof, we blur the lines between fundamental quandle elements and their path representatives.

	First, we establish some notation.
	Let $N$ be a normal disk bundle about $L$ in $M$ and let $N_U = N \cap U$, $N_V = N \cap V$ and
	$N_{U \cap V} = N \cap U \cap V$. Denote also their respective exteriors by $E =  \cl(M - N)$,
	$E_U =E \cap U$, $E_V = E \cap V$ and
	$E_{U \cap V} = E \cap U \cap V$. Pick a basepoint $z \in E_{U \cap V}$.

	To begin, we define a suitable quandle homomorphism 
	$$F: Q(U, L \cap U)*Q(V, L \cap V) \to Q(M, L).$$
	Given representations of $Q(U, L \cap U)$ and $Q(V, L \cap V)$, we interpret a formal product
	of generators in $Q(U, L \cap U)*Q(V, L \cap V)$ as a path in $E$ from $\partial N$ to $z$:
	by definition of the fundamental quandle, each generator $\alpha$ from $Q(U, L \cap U)$ is represented by some path in $E_U$
	from $\partial N_U$ to $z$, and each generator $\beta$ from $Q(V, L \cap V)$ is represented by some path in $E_V$
	from $\partial N_V$ to $z$. These paths also represent some elements of $Q(M,L)$, so their formal products $\alpha \tr \beta$ and
	$\alpha \tl \beta$ can be interpreted as the corresponding products in the fundamental quandle $Q(M,L)$, 
	giving us a well-defined quandle homomorphism $F$.

	Next, we prove that $F$ is in fact an epimorphism.
	Let $\alpha: I \to E$ be a path representing an element in $Q(M, L)$, meaning that $\alpha(0) \in \partial N$
	and $\alpha(1) = z$. We can assume, up to homotopy keeping $\alpha(0) \in \partial N$ and $\alpha(1) = z$, that $\alpha(0) \in \partial N_{U \cap V}$
	and that there exists some $t_0 \in I$ such that $\alpha(t_0) = z$ and $\alpha_0 := \alpha([0, t_0]) \subset E_{U \cap V}$.
	Then $\alpha([t_0, 1])$ is some loop starting and ending at $z$ and therefore represents some element in $\pi_1(E)$. Notice
	that the spaces $E = E_U \cup E_V$ and $E_{U \cap V} = E_U \cap E_V$ satisfy the conditions for the Seifert-Van-Kampen theorem for groups.
	We can thus decompose the loop $\alpha([t_0, 1])$ as a concatenation of loops
	$\beta_1 \dots \beta_k$, where each loop $\beta_i$ represents an element of $\pi_1(E_U)$ or $\pi_1(E_V)$. As $U$ and $V$ are simply connected,
	these groups are generated by the meridians around $L$, and we can therefore decompose each $\beta_i$ as a product of these meridians,
	which can be chosen to be of the form
	$\gamma^{-1} m^{\pm 1} \gamma$, where $\gamma$ is some path from $\partial N_U$ or $\partial N_V$ to $z$ (depending on if
	the loop belongs to $\pi_1(E_U)$ or $\pi_1(E_V)$) and $m$ is some positively oriented meridian in 
	$\partial N_U$ or $\partial N_V$, based at the beginning of $\gamma$. Altogether
	we have decomposed $\alpha$ as a concatenation of paths
	$$\alpha = \alpha_0 \cdot ({\alpha_1}^{-1} m_1^{\epsilon_1} \alpha_1) \cdot ({\alpha_2}^{-1} m_2^{\epsilon_2} \alpha_2)\cdot \ldots \cdot({\alpha_n}^{-1} m_n^{\epsilon_n} \alpha_n),$$
	where each $\alpha_i$ is a path starting at $\partial N_U$ or $\partial N_V$ and ending at $z$, 
	each $m_i$ is a positively oriented meridian in $\partial N_U$ or $\partial N_V$, which is based at the beginning of $\alpha_i$, and
	$\epsilon_1, \dots, \epsilon_n \in \{-1, 1\}$. By definition of the fundamental quandle of a codimension two embedding, we then have a quandle factorization in $Q(M,L)$
	$$\alpha = ( \dots ((\alpha_0 \diamond^{\epsilon_1} \alpha_1) \diamond^{\epsilon_2} \alpha_2) \diamond^{\epsilon_3} \dots ) \diamond^{\epsilon_n} \alpha_n,$$
	where $\diamond^{1} = \tr$, $\diamond^{-1} = \tl$ and each $\alpha_i$ represents some element of
	$Q(U, L \cap U)$ or $Q(V, L \cap V)$. It follows that $\alpha$
	is the image under $F$ of the corresponding product in $Q(U, L \cap U)*Q(V, L \cap V)$, showing that $F$ is in fact an epimorphism.

	Finally, we prove that two elements of $Q(U, L \cap U)*Q(V, L \cap V)$ get mapped to the same element of $Q(M, L)$
	if and only if they are related by the congruence $\sim _Q$, generated by the relations $\iota_U(\alpha) = \iota_V(\alpha)$ for all $\alpha \in Q(U \cap V, L \cap U \cap V)$.
	On implication is easy: we can interpret the congruence by looking at its effect on the path representatives of elements in $Q(U, L \cap U)*Q(V, L \cap V)$. The
	congruence doesn't change the actual paths; it only changes which quandle ($Q(U, L \cap U)$ or $Q(V, L \cap V)$)
	the path belongs to. It therefore doesn't change the image under $F$, which is defined via inclusions of paths, 
	meaning that congruent elements get mapped to the same element. Then $F$ is actually a well-defined epimorphism
	$$F: Q(U, L \cap U)*_{Q(U \cap V, L \cap U \cap V)}Q(V, L \cap V) \twoheadrightarrow Q(M, L).$$

	Let now $\alpha$ and $\beta$ be two elements of $Q(U, L \cap U)*Q(V, L \cap V)$ which get mapped by $F$ to the same element of $Q(M,L)$. Abusing notation,
	we denote the paths representing $F(\alpha)$ and $F(\beta)$ by $\alpha$ and $\beta$ as well. As when proving that $F$ is an epimorphism, decompose
	$\alpha$ and $\beta$ as the concatenations of paths
	\begin{align*}
		\alpha &= \alpha_0 \cdot ({\alpha_1}^{-1} m_1^{\epsilon_1} \alpha_1) \cdot ({\alpha_2}^{-1} m_2^{\epsilon_2} \alpha_2)\cdot \ldots \cdot({\alpha_k}^{-1} m_k^{\epsilon_k} \alpha_k) \text{ and}\\
		\beta &= \beta_0 \cdot ({\beta_1}^{-1} n_1^{\eta_1} \beta_1) \cdot ({\beta_2}^{-1} n_2^{\eta_2} \beta_2)\cdot \ldots \cdot({\beta_l}^{-1} n_l^{\eta_l} \beta_l),
	\end{align*}
	where $\alpha_0, \dots, \alpha_k, \beta_0, \dots, \beta_l$ are paths from $\partial N_U$ or $\partial N_V$ to $z$, 
	each $m_i$ and each $n_j$ is a positively oriented meridian in $\partial N_U$ or $\partial N_V$ based at the beginning of $\alpha_i$ and $\beta_i$, respectively,
	and $\epsilon_1, \dots, \epsilon_k, \eta_1, \dots, \eta_l \in \{-1, 1\}$. Since $\alpha$ and $\beta$ represent the same element of $Q(M,L)$, there
	exists a homotopy $H: I_t \times I_s \to E$ such that $H(t,0) = \alpha(t)$, $H(t,1) = \beta(t)$, $H(0, s) \in \partial N$ and $H(1,s) = z$ for all $t \in I_t$ and $s \in I_s$.
	As $I_t \times I_s$ is compact, there
	exist $0 = t_0 < t_1 < \dots < t_p = 1$ and $0 = s_0 < s_1 < \dots < s_r = 1$ such that $F([t_{i-1}, t_i] \times [s_{j-1}, s_j])$ is completely contained in $E_U$ or $E_V$ for all
	$i \in \{1, \dots, p\}$ and $j \in \{1, \dots, r\}$. We may additionally assume that $\{t_0, \dots, t_p\}$ subdivides the partitions splitting $\alpha$ and $\beta$
	as the concatenations of a path $\alpha_0$ or $\beta_0$ and loops ${\alpha_i}^{-1} m_i^{\epsilon_i} \alpha_i$ or ${\beta_j}^{-1} n_j^{\eta_j} \beta_j$, described above. 
	Order the rectangles $[t_{i-1}, t_i] \times [s_{j-1}, s_j]$ linearly as $R_1 < \dots < R_{pr}$ by the rule
	$$([t_{i-1}, t_i] \times [s_{j-1}, s_j]) < ([t_{i'-1}, t_{i'}] \times [s_{j'-1}, s_{j'}]) \iff (j < j') \text{ or } (j = j', i < i').$$
	Let $\gamma_q$ be the injective path in $I_t \times I_s$ from $\{0\}_t \times I_s$ to $\{1\}_t \times I_s$, starting and ending on corners and otherwise lying 
	entirely on the edges of the rectangles $R_1, \dots, R_{pr}$, whose image
	separates the first $q$ rectangles from the remaining $pr -q$ rectangles. 
	Let also $\gamma_0(t) = (t,0)$ and $\gamma_{pr}(t) = (t,1)$ be the top and bottom edges of  $I_t \times I_s$. Then $H(\gamma_q(t))$ is a path from $\partial N$ to $z$
	for each $q \in \{0, \dots, pr\}$, meaning it represents an element of $Q(M,L)$.
	For any corner $v$ of any rectangle $R_{ij}$ such that $H(v) \neq z$, we can pick a path $\delta_v$ from $H(v)$ to $z$ lying entirely in $E_U$ or $E_V$
	(if all the rectangles that border $v$ get mapped to $E_U$ or $E_V$, respectively) or lying entirely in $E_{U \cap V}$ (if $v$ borders some rectangle that gets mapped to $E_U$ and some
	rectangle that gets mapped to $E_V$). By inserting homotopically trivial loops $\delta_v \delta_v^{-1}$ into these
	corners, we then decompose each $H(\gamma_q)$ into a concatenation of a path from $\partial N$ to $z$, followed by loops lying entirely in $E_U$ or $E_V$ or $E_{U\cap V}$. 
	As $U$, $V$ and $U \cap V$
	are simply connected, we can assume these loops are products of loops of the form $\gamma^{-1} m^{\pm 1} \gamma$, where $\gamma$ is a path from $\partial N$ to $z$ and $m$ is a
	positively oriented meridian around $L$, based at the start of $\gamma$. These
	decompositions then give us quandle factorizations in $Q(U, L \cap U)*Q(V, L \cap V)$. This factorization is not unique, as a loop in $E_{U \cap V}$ may either represent
	an element of $Q(U, L \cap U)$ or $Q(V, L \cap V)$, but is unique up to the congruence $\sim _Q$, which exactly identifies such elements. 
	Moreover, by pushing $H$ over $R_q$, we obtain a homotopy between $H(\gamma_{q-1})$ and $H(\gamma_q)$ lying entirely in either $E_U$ or $E_V$, depending on where $H$ maps $R_q$.
	By construction, this homotopy induces an equivalence of the quandle factorizations of $H(\gamma_{q-1})$ and $H(\gamma_1)$ up to $\sim _Q$ in $Q(U, L \cap U)*Q(V, L \cap V)$.
	Lastly, notice that the original quandle factorizations obtained from the decompositions of $\alpha$ and $\beta$ are equivalent, up to $\sim _Q$, to the quandle
	factorizations of $H(\gamma_0)$ and $H(\gamma_{pr})$, respectively. We have then proven that $\alpha = \beta$ as elements of $Q(U, L \cap U)*_{Q(U \cap V, L \cap U \cap V)}Q(V, L \cap V)$, meaning
	that $F$ is injective and therefore an isomorphism.
\end{proof}

\begin{remark}
	Note that Joyce already gave a similar theorem for the augmented quandle in~\cite{J}, but that that theorem doesn't
	hold in full generality for the fundamental quandle without any additional structure. The problem arises from the fact that
	without the additional structure given by the augmentation, it might not be possible to decompose every path representing an element of $Q(M,L)$
	as a quandle product of paths in $Q(U, L \cap U)$ and $Q (V, L \cap V)$.
	We avoid this issue by requiring that the spaces $U$, $V$ and $U \cap V$ are simply connected, and
	as such the fundamental groups $\pi_1(E_U)$, $\pi_1(E_V)$ and $\pi_1(E_{U\cap V})$ are
	generated by meridians around $L$. The injectivity part of our proof then closely follows the injectivity proof of the regular Seifert-Van-Kampen theorem
	(see for example the relevant subchapter in~\cite{H}).

\end{remark}

\begin{remark}
	Theorem~\ref{thm:svk} generalizes for any open covering $\{U_i\}_{i \in \Lambda}$ of $(M,L)$ such that
	\begin{itemize}
		\item $U_i - L$, $U_i \cap U_j - L$ and $U_i \cap U_j \cap U_k - L$ are path-connected and
		\item $U_i$, $U_i \cap U_j$, $U_i \cap U_j \cap U_k$ are simply connected
	\end{itemize}
	\noindent for any $i,j,k \in \Lambda$.
\end{remark}

\section{The fundamental quandle of an embedded surface}
\label{sec3}

A \textit{surface-link} is a properly smoothly embedded surface in some $4$-manifold. Surface-links with one component are called \textit{surface-knots}.
Closed surfaces in 4-space have been studied via a plethora of diagrammatic theories, such as broken surface diagrams, motion pictures,
banded unlink diagrams, marked vertex diagrams, chord diagrams and more~\cite{CS1,CKS2}.
The fundamental quandle of a closed knotted surface represented by a broken surface diagram was described in~\cite{CKS1}. 
Since we are interested in the fundamental quandle of a surface with boundary and would also like to relate it with the fundamental quandles of its boundary components,
a motion picture of a knotted surface will suit our purposes better. Let us first describe a motion picture presentation of a knotted surface.

Let $\Sigma \subset \RR ^{3} \times I$ be a compact orientable properly embedded surface and let $\pi \colon \RR ^{3} \times I\to I $ be the projection onto the second factor.
We may assume that $\pi $ restricts to a Morse function on $\Sigma$. 
Note that the regular fibres of $\pi|_{\Sigma}$ are then classical links in $\RR^3$, 
while every singular fibre is (generally) an immersed link with some isolated points in $\RR^3$. 
A \textit{motion picture} of $\Sigma$ is the collection $\bigl\{ \pi|_{\Sigma}^{-1}(t) \subset \RR^3 \times \{t\}\bigr\}_{t \in I}$
of all fibres of $\pi |_{\Sigma }$, called stills.
It is common practice to represent a motion picture of a surface by just a finite sequence of classical link diagrams of stills
$(D_{1},D_{2},\ldots ,D_{n})$, where two subsequent diagrams differ by a planar isotopy or by 
exactly one of the local moves depicted in figure~\ref{fig:ESI}. 
The first three transformations correspond to the passage over a critical point of the Morse function $\pi|_{\Sigma}$, 
while the last three correspond to the changes of link diagrams that may occur inside regular intervals.
\begin{itemize}
	\item A critical point of index 0 corresponds to the addition of a 0-handle, which gives birth to a new component in the still (transformation (i)).
	\item A critical point of index 1 corresponds to the addition of a 1-handle, changing the still by surgery along a band (transformation (ii)). 
	\item A critical point of index 2 corresponds to the addition of a 2-handle along an unknotted still component, causing its death (transformation (iii)).
	\item Inside regular intervals, the ambient isotopy classes of the stills are constant, meaning that their diagrams can change by the Reidemeister moves (transformations (iv) through (vi)) or by planar isotopy.
\end{itemize}
While presenting a motion picture in such a way, however, one must be careful to record planar isotopy between two subsequent stills, 
as forgetting that can lead to disastrous results, 
such as depicting two non-ambient-isotopic surfaces with the same sequence of diagrams; 
see for example figure~16 and figure~17 and the discussion on page~17 of~\cite{CS1}.

\begin{figure}[H]
	\centering
	\begin{tikzpicture}
		\matrix
		{
		& \node {
			\begin{tikzpicture}[scale=0.6]
				\draw[black,thick] (0,2.5) node[anchor=center]{(i)};
				\draw (-1.5,-1.5) rectangle (1.5,1.5);
				\draw (-1.5,-6.5) rectangle (1.5,-3.5);
				\draw[white,thick] (0,0) circle (1cm);
				\draw[black, thick][->] (0,-2) -- (0,-3);
				\draw[red,thick] (0,-5) circle (1cm);
			\end{tikzpicture}}; &
		\node {
			\begin{tikzpicture}[scale=0.6]
				\draw[black,thick] (0,2.5) node[anchor=center]{(ii)};
				\draw (-1.5,-1.5) rectangle (1.5,1.5);
				\draw (-1.5,-6.5) rectangle (1.5,-3.5);
				\draw[red, thick] (-1,1) .. controls (0,0.4) and (0,-0.4) .. (-1,-1) ;
				\draw[red, thick] (1,1) .. controls (0,0.4) and (0,-0.4) .. (1,-1);
				\draw[black, thick][->] (0,-2) -- (0,-3);
				\draw[red, thick] (1,-4) .. controls (0.4,-5) and (-0.4,-5) .. (-1,-4) ;
				\draw[red, thick] (1,-6) .. controls (0.4,-5) and (-0.4,-5) .. (-1,-6);
			\end{tikzpicture}}; & 
		\node {
			\begin{tikzpicture}[scale=0.6]
				\draw[black,thick] (0,2.5) node[anchor=center]{(iii)};
				\draw (-1.5,-1.5) rectangle (1.5,1.5);
				\draw (-1.5,-6.5) rectangle (1.5,-3.5);
				\draw[red,thick] (0,0) circle (1cm);
				\draw[black, thick][->] (0,-2) -- (0,-3);
				\draw[white,thick] (0,-5) circle (1cm);
			\end{tikzpicture}}; & 
		\node {
			\begin{tikzpicture}[scale=0.6]
				\draw[black,thick] (0,2.5) node[anchor=center]{(iv)};
				\draw (-1.5,-1.5) rectangle (1.5,1.5);
				\draw (-1.5,-6.5) rectangle (1.5,-3.5);
				\draw[red, thick] (-1,1) .. controls (0,0.4) and (0,-0.4) .. (-1,-1) ;
				\draw[black, thick][->] (0,-2) -- (0,-3);
				\draw[red, thick] (-1,-4) .. controls (1.7,-6.5) and (1.7,-4) .. (0.2,-4.8) ;
				\draw[red, thick] (0,-5) .. controls (-0.8,-5.7) .. (-1,-6);
			\end{tikzpicture}}; & 
		\node {
			\begin{tikzpicture}[scale=0.6]
				\draw[black,thick] (0,2.5) node[anchor=center]{(v)};
				\draw (-1.5,-1.5) rectangle (1.5,1.5);
				\draw (-1.5,-6.5) rectangle (1.5,-3.5);
				\draw[red, thick] (-1,1) .. controls (0,0.4) and (0,-0.4) .. (-1,-1) ;
				\draw[red, thick] (1,1) .. controls (0,0.4) and (0,-0.4) .. (1,-1);
				\draw[black, thick][->] (0,-2) -- (0,-3);
				\draw[red, thick] (-0.7,-4) .. controls (0.6,-4.6) and (0.6,-5.4) .. (-0.7,-6) ;
				\draw[red, thick] (0.1,-4.4) .. controls (0.5,-4.1) .. (0.6,-4);
				\draw[red, thick] (-0.2,-4.5) .. controls (-0.7,-4.7) and (-0.7,-5.3) .. (-0.2,-5.5);
				\draw[red, thick] (0.1,-5.6) .. controls (0.2,-5.6) .. (0.6,-6);
			\end{tikzpicture}}; & 
		\node {
			\begin{tikzpicture}[scale=0.6]
				\draw[black,thick] (0,2.5) node[anchor=center]{(vi)};
				\draw (-1.5,-1.5) rectangle (1.5,1.5);
				\draw (-1.5,-6.5) rectangle (1.5,-3.5);
				\draw[red, thick] (-1,-1) .. controls (-0.7,-0.9) .. (-0.5,-0.4) ;
				\draw[red, thick] (-0.4,-0.2) .. controls (-0.1,0.5) .. (1,1);
				\draw[red, thick] (-1,0) .. controls (-0.2,-0.5) and (0.2,-0.5) .. (1,0);
				\draw[red, thick] (-1,1) .. controls (-0.7,0.6) .. (-0.2,0.4);
				\draw[red, thick] (1,-1) .. controls (0.7,-0.9) .. (0.5,-0.4) ;
				\draw[red, thick] (0.4,-0.2) .. controls (0.3,0.3) .. (0.1,0.4) ;
				\draw[black, thick][->] (0,-2) -- (0,-3);
				\draw[red, thick] (-1,-5) .. controls (-0.2,-4.3) and (0.2,-4.3) .. (1,-5);
				\draw[red, thick] (-1,-6) .. controls (-0.7,-5.9) .. (0.4,-4.7) ;
				\draw[red, thick] (0.5,-4.5) .. controls (0.7,-4.1) .. (1,-4);
				\draw[red, thick] (-1,-4) .. controls (-0.7,-4.1) .. (-0.5,-4.5);
				\draw[red,thick] (-0.4,-4.65) .. controls (-0.2,-4.9) ..  (-0.1,-5);
				\draw[red, thick] (1,-6) .. controls (0.7,-5.9) .. (0.05,-5.2) ;
			\end{tikzpicture}}; \\
		};
	\end{tikzpicture}
	\caption{Motion picture transformations.}
	\label{fig:ESI}
\end{figure}

Note that figure~\ref{fig:ESI} depicts the moves for unoriented surfaces. When working with oriented surfaces, as we will, one additionally requires that the moves
respect the orientations of the stills, which are induced by the orientation of the surface.

The goal of the remainder of this section is to provide an algorithm for computing the fundamental quandle of a surface-link from its motion picture. 
To do so, we will need a handful of simple lemmas.

\begin{lemma} 
	\label{lemma1} 
	Let $L$ be a link in $\RR^3$. 
	Consider $L\times I$ as a properly embedded surface in $\RR^3 \times I$. 
	The inclusion of pairs $i:(\RR^3 \times \{0\},L)\hookrightarrow (\RR^3 \times I,L\times I)$ induces an isomorphism
	$Q(\RR^3,L) \cong Q(\RR^3 \times I,L\times I)$.
\end{lemma}
\begin{proof}
	Let $N_L \subset \RR^3 \times \{0\}$ and $N_{L \times I} = N_L \times I \subset \RR^3 \times I$ be normal disk bundles of $L$ and $L\times I$, respectively, and
	let $E_L$ and $E_{L \times I} $ be their respective exteriors.
	Choose a basepoint $z \in E_L\subset E_{L \times I}$ that defines the fundamental quandles $Q(\RR^3, L)$ and $Q(\RR^3 \times I, L \times I)$.
	Every path in $E_L$ from $\partial N_L$ to $z$ is also a path in $E_{L \times I}$ from $\partial N_{L \times I}$ to $z$, 
	and every homotopy identifying two representatives of an element of 
	$Q(\RR^3, L)$ identifies these representatives in $Q(\RR^3 \times I, L\times I)$. 
	Then the inclusion $i:(\RR^3 \times \{0\},L)\hookrightarrow (\RR^3 \times I,L\times I)$
	induces a well-defined mapping $\iota: Q(\RR^3, L) \to Q(\RR^3 \times I, L \times I)$. 
	It follows directly from the definition that $\iota$ is in fact a quandle homomorphism.
	
	Suppose $\alpha$ and $\beta$ are two paths representing elements of $Q(\RR^3, L)$ such that $\iota([\alpha]) = \iota([\beta])$,
	meaning that there exists a homotopy $H\colon I \times I \to E_{L \times I} \subset \RR^3 \times I$ between 
	$\alpha$ and $\beta$ with $H(\{0\}\times I)\subset \partial N_{L\times I}$ and $H(\{1\}\times I)=\{z\}$. 
	Let $H(t,s) = (H_{\RR^3}(t,s), H_I(t,s))$ where $H_{\RR^3}(t,s) \in \RR^3$ and $H_I(t,s) \in I$. 
	Then $\overline{H}: I \times I \to E_L$, defined by $\overline{H}(t,s) = (H_{\RR^3}(t,s), 0)$, 
	is a homotopy between $\alpha$ and $\beta$ in $E_L$ with $\overline{H}(\{0\}\times I)\subset \partial N_{L}$ and $\overline{H}(\{1\}\times I) = \{z\}$. 
	Therefore, $\alpha$ and $\beta$ represent the same element of $Q(\RR^3, L)$, showing that $\iota$ is injective.

	Given a path $\gamma(t) = (\gamma_{\RR^3}(t), \gamma_I(t)) \subset E_{L\times I} \subset \RR^3 \times I$ with 
	$\gamma_{\RR^3}(t) \in \RR^3$ and 
	$\gamma_I(t) \in I$, representing an element of $Q(\RR^3 \times I, L \times I)$,
	we can define a homotopy $H_{\gamma}:I \times I \to E_{L\times I}$ from $\gamma$ to a path representing an element of $Q(\RR^3, L)$ by
    $H_{\gamma}(t,s) = \left (\gamma_{\RR^3}(t), (1-s)\gamma_I(t)\right )$,
	showing that $\iota$ is surjective and thus an isomorphism.
\end{proof}

Until the end of the section, we assume that $\Sigma \subset \RR^3\times I$ 
is a compact oriented properly embedded surface and that the projection 
$\pi: \RR^3 \times I \to I$ to $\Sigma$ is Morse on $\Sigma$.
Let $\Sigma_J = \Sigma \cap \RR^3 \times J$ for any subinterval $J \subset I$ and let
$L_a = \Sigma \cap \RR^3 \times \{a\}$ for any $a \in I$. We denote by
$i$ the inclusion of pairs $(\RR^3 \times \{0\}, L_0 ) \hookrightarrow (\RR^3 \times I, \Sigma)$ and by
$\iota$ the induced quandle homomorphism $Q(\RR^3, L_0) \to Q(\RR^3 \times I, \Sigma)$. 

\begin{lemma}
	\label{lemma2}
	If $\pi$ has no critical points on $\Sigma$, then $\iota$ is an isomorphism.
\end{lemma}
\begin{proof}
	Since $\pi |_{\Sigma }$ is regular, $\Sigma $ is ambient isotopic to $L_{0}\times I$ (rel $\RR^{3}\times {0}$) 
	by lemma 3.3.5 in~\cite{K}, where it is called the
	cylindrical form lemma. As the fundamental quandle is invariant under ambient isotopy, it follows that $Q(\RR^3 \times I, \Sigma) \cong Q(\RR^3 \times I, L_0 \times I)$.
	By lemma~\ref{lemma1} we have that $Q(\RR^3,L_{0}) \cong Q(\RR^3 \times I,L_{0}\times I)$, completing the proof.
\end{proof}

\begin{lemma}
	\label{lemma3}
	If $\pi$ has exactly one critical point,
	and the index of that critical point equals $0$, then $\iota$ is a monomorphism. Moreover, it is realized as the natural inclusion of
	$Q(\RR^3, L_0)$ into the free product of $Q(\RR^3, L_0)$ and a trivial quandle.
\end{lemma}
\begin{proof}
	Let $t$ be the critical value. Then there exists a small enough $\epsilon > 0$ such that
	$\Sigma_{[ t - \epsilon, t + \epsilon]}$ is a split surface-link composed of a disk $D$
	with boundary $\partial D$ at height $t + \epsilon$ and a 
	cylindrical part $L_{t - \epsilon} \times [ t - \epsilon, t + \epsilon]$. 
	By lemma~3.3.5 in~\cite{K}, $\Sigma$ is ambient isotopic to the split surface link
	$$\Sigma' = (L_{t - \epsilon} \times I) \sqcup D',$$
	where $D' := (D \cup (\partial D \times [t + \epsilon, 1]))$.
	A use of the Seifert-Van-Kampen theorem for quandles (theorem~\ref{thm:svk}) gives us
	$$Q(\RR^3 \times I, \Sigma') \cong Q(\RR^3 \times I, L_{t - \epsilon} \times I) * Q(\RR^3 \times I, D').$$
	By lemma~\ref{lemma1}, the inclusion $L_{t - \epsilon} \times \{0\} \hookrightarrow L_{t - \epsilon} \times I$ 
	induces a quandle isomorphism 
	$Q(\RR^3, L_{t - \epsilon}) \cong Q(\RR^3 \times I, L_{t - \epsilon} \times I)$, 
	meaning that it induces a quandle monomorphism $Q(\RR^3, L_{t - \epsilon}) \hookrightarrow Q(\RR^3 \times I, \Sigma')$.
	Moreover, $D'$ is an unknotted disk, as it has only one critical point (see theorem~3.3.1 of~\cite{K}), 
	and thus has a trivial fundamental quandle with one generator $x$.
	Then $Q(\RR^3 \times I, \Sigma') \cong Q(\RR^3, L_{t - \epsilon}) * \langle x \rangle$. The ambient isotopy moving $\Sigma'$ to $\Sigma$
	then shows that $\iota$ is a monomorphism of the desired form.

\end{proof}

\begin{lemma}
	\label{lemma4}
	If $\pi$ has exactly one critical point, and the index of that critical point equals $1$, then
	$\iota$ is an epimorphism.
	Moreover, it is realized by adding the relation identifying the two elements of $Q(\RR^3, L_0)$,
	represented by the arcs that the 1-handle attaches to.
\end{lemma}
\begin{proof}
	We deform the neighbourhood of the critical point to a critical band $B$ at height $t$.
	There then exists some $\epsilon > 0$ such that
	$\Sigma_{[ t - \epsilon, t)}$ and $ \Sigma_{( t, t + \epsilon]}$
	are ambient isotopic to
	$$L_{t - \epsilon} \times [ t - \epsilon, t) \quad \text{and} \quad L_{t + \epsilon} \times ( t, t + \epsilon],$$
	respectively, where $B$ is attached to $L_{t - \epsilon}$ at height $t$ and $L_{t + \epsilon}$ is the result of surgery
	on $L_{t-\epsilon}$ along $B$. Then $\Sigma$ is ambient isotopic to the surface
	$$\Sigma' = \Bigl(L_{t - \epsilon} \times [0, t)\Bigr) \cup \Bigl(\bigl(L_{t - \epsilon} \cup B \bigr)\times \{t\}\Bigr)
	\cup \Bigl(L_{t + \epsilon} \times (t, 1]\Bigr).
	$$
	To compute $Q(\RR^3 \times I, \Sigma')$, notice that the surface is cylindrical everywhere, 
	except in a neighbourhood of the band at height $t$, 
	and that the band allows paths from one side to wander to the other. This identifies paths starting at different sides of the bands,
	adding a relation equating the quandle elements represented by the two arcs. Note also that no new generators are added, as
	any path starting in a neighbourhood of the band itself can be homotopied to a path starting in the cylindrical part of the surface. Then
	the inclusion $(\RR^3 \times \{0\}, L_{t - \epsilon}) \hookrightarrow (\RR^3 \times I, \Sigma')$ induces an
	epimorphism $Q(\RR^3, L_{t - \epsilon}) \twoheadrightarrow Q(\RR^3 \times I, \Sigma')$,
	which is realized by adding the described relation. 
	The ambient isotopy moving $\Sigma'$ to $\Sigma$ then shows that $\iota$ is an epimorphism of the desired form.
\end{proof}

\begin{lemma}
	\label{lemma5}
	If $\pi$ has exactly one critical point,
	and the index of that critical point equals $2$, then $\iota$ is an isomorphism.
\end{lemma}
\begin{proof}
	We may assume that the critical value is $t$.
	Similarly as in the proof of lemma~\ref{lemma3}, there exists a small enough $\epsilon > 0$ such that 
	$\Sigma_{[t-\epsilon, t + \epsilon]}$ is a split surface link composed of a disk $D$
	with boundary $\partial D$ at height $t - \epsilon$ and a cylindrical part
	$L_{t + \epsilon} \times [ t - \epsilon, t + \epsilon]$. Then, again by lemma~3.3.5 of~\cite{K},
	$\Sigma$ is ambient isotopic to the split surface-link
	$$\Sigma' = (L_{t + \epsilon} \times I) \sqcup D',$$
	where $D' := D \cup (\partial D \times [0, t - \epsilon])$.
	By the Seifert-Van-Kampen theorem for quandles (theorem~\ref{thm:svk}):
	$$Q(\RR^3 \times I, \Sigma') \cong Q(\RR^3 \times I, L_{t + \epsilon} \times I) * Q(\RR^3 \times I, D').$$
	Note that $D'$ is an unknotted disk, as it has only one critical point (see theorem~3.3.1 of~\cite{K}),
	so $Q(\RR^3 \times I, D')$ is a trivial quandle.
	As $L_{t + \epsilon} \sqcup \partial D$ is a split link, we also have
	$$Q(\RR^3, L_{t + \epsilon} \sqcup \partial D) \cong Q(\RR^3, L_{t + \epsilon}) * Q(\RR^3, \partial D) \cong Q(\RR^3, L_{t + \epsilon}) * \langle x \rangle .$$
	Then the inclusion $L_{t + \epsilon} \sqcup \partial D \hookrightarrow \Sigma$ consists of the inclusions $L_{t + \epsilon} \hookrightarrow L_{t + \epsilon} \times I$
	and $\partial D \hookrightarrow D'$, both of which induce quandle isomorphisms (the former by lemma~\ref{lemma1}, the latter
	as it is a homomorphism between trivial quandles). The ambient isotopy taking
	$\Sigma'$ to $\Sigma$ then shows that $\iota$ is an isomorphism.
\end{proof}

The lemmas above all deal with quandle homomorphisms induced by inclusions of a boundary component into a surface. To use the Seifert-Van-Kampen theorem, we will
generally rather need to deal with homomorphisms induced by inclusions of open intervals. The next lemma helps us square this circle.

\begin{lemma}
	\label{lemma6}
	Let $a, b \in (0,1)$, $a < b$ be regular values with regard to $\pi|_\Sigma$. Let $\epsilon > 0$ be such that $[a- \epsilon, a + \epsilon]$
	and $[b- \epsilon, b + \epsilon]$ contain no critical values of $\pi|_\Sigma$.
	Then we have a commutative square of quandle homomorphisms induced by inclusions:
	\begin{figure}[H]
		\centering
		\begin{tikzcd}
			Q(\RR^3 \times \{a\}, L_a) 												\ar[r] \ar[d, "\cong"']	&Q(\RR^3 \times [a, b], \Sigma_{[a, b]}) \ar[d, "\cong"]\\
			Q(\RR^3 \times (a - \epsilon, a + \epsilon), \Sigma_{(a-\epsilon, a+\epsilon)}) \ar[r]					&Q(\RR^3 \times (a - \epsilon, b + \epsilon), \Sigma_{(a-\epsilon, b+\epsilon)})
		\end{tikzcd}
	\end{figure}
\end{lemma}
\begin{proof}
	First, we prove that the vertical maps on the diagram above are isomorphisms. To do so, we note that $\Sigma_{(a-\epsilon, b + \epsilon)}$ is ambient isotopic to the surface
	$\Sigma'_{(a-\epsilon, b + \epsilon)}$, given by cross-sections
	$$
	\Sigma'_{(a-\epsilon, b + \epsilon)} \cap \RR^3 \times \{t\} =
	\begin{cases}
		L_a;& t \in (a - \epsilon, a],\\
		L_t;& t \in (a,b),\\
		L_b;& t \in [b, b+ \epsilon)
	\end{cases}
	$$
	(use lemma~3.3.5 of~\cite{K} for the intervals $[a - \epsilon, a]$ and $[b, b+ \epsilon]$ then remove the bottommost and topmost level, respectively).
	Then the inclusion $[a,b] \hookrightarrow (a- \epsilon, b + \epsilon)$ induces a quandle isomorphism
	between $Q(\RR^3 \times [a, b], \Sigma_{[a, b]})$ and $Q(\RR^3 \times (a - \epsilon, b + \epsilon), \Sigma_{(a-\epsilon, b+\epsilon)})$ by a similar argument as in the proof of lemma~\ref{lemma1};
	we check manually that any homotopy in $\pi^{-1}(a- \epsilon, b+ \epsilon)$ between paths starting in $\pi^{-1}[a,b]$ can be pushed to a homotopy wholly in $\pi^{-1}[a,b]$, and that
	any path in $\pi^{-1}(a - \epsilon, b + \epsilon)$ can be homotoped to a path in $\pi^{-1}[a,b]$. Similarly, $L_a \hookrightarrow \Sigma_{(a - \epsilon, a + \epsilon)}$
	also induces an isomorphism.

	To see that the diagram commutes, we can look at the maps at the level of path representatives of quandle elements (meaning we look at representatives of homotopy
	classes of paths). Since the maps on these paths are all inclusions, the diagram clearly commutes.
\end{proof}

\begin{remark}
	\label{remark1}
	A similar statement to lemma~\ref{lemma6} holds for half-open intervals; we have
	$$Q(\RR^3 \times [a,b], \Sigma_{[a,b]}) \cong Q(\RR^3 \times [a, b+\epsilon), \Sigma_{[a,b+\epsilon)})$$
	and that the quandle homomorphism induced by $L_a \hookrightarrow \Sigma_{[a,b]}$ is the same as the one induced by
	$\Sigma_{(a - \epsilon, a]} \hookrightarrow \Sigma_{(a - \epsilon, b]}$.
\end{remark}

Lastly, the following lemma will simplify the argument on the well-de\-fi\-ne\-dness of the algorithm with respect to ambient isotopy.

\begin{lemma}
	\label{lemma7}
	Any ambient isotopy between two surfaces in $\RR^3 \times I$
	can be realized as a composition of ambient isotopies supported in either a collar neighbourhood
	of the boundary of $\RR^3 \times I$ or in the interior of $\RR^3 \times I$.
	Moreover, it can be assumed that all the ambient isotopies supported in a collar neighbourhood of the boundary
	fix all critical points.
\end{lemma}
\begin{proof}
	Let $$\bigl\{\phi_t: \RR^3 \times I \to \RR^3 \times I\bigr\}_{t \in [0,1]}$$ 
	be an ambient isotopy between $\Sigma$ and $\Sigma'$. 
	Define $\Sigma_s = \Sigma \cap \RR^3 \times \{s\}$ and $\Sigma'_s = \Sigma' \cap \RR^3 \times \{s\}$ for all $s \in I$.
	Let $0 < \epsilon < \tfrac{1}{6}$ be such that 
	$$\phi_t(\Sigma) \cap \bigl(\RR^3 \times ([0, 3 \epsilon] \cup [1 - 3 \epsilon, 1])\bigr)$$
	contains no critical points with respect to the projection $\pi$ for all $t \in [0,1]$. 
	Define surfaces $\Sigma^a, \Sigma^b$ and $\Sigma^c$, given by intersections
	\begin{align*}
		\Sigma^a \cap \RR^3 \times \{s\} &=
		\begin{cases}
			\phi_{1 - \tfrac{s }{\epsilon}}(\Sigma_0);& s \in [0, \epsilon], \\
			\Sigma_0; & s \in [\epsilon, 2 \epsilon], \\
			\Sigma_{3s - 6 \epsilon};& s \in [2 \epsilon, 3 \epsilon], \\
			\Sigma_s; & s \in [3 \epsilon, 1 - 3 \epsilon],\\
			\Sigma_{3s -2 + 6\epsilon};& s \in [1- 3\epsilon, 1- 2 \epsilon],\\
			\Sigma_1; & s \in [1 - 2 \epsilon, 1 - \epsilon], \\
			\phi_{1 - \tfrac{1-s}{\epsilon}}(\Sigma_1); & s \in [1 - \epsilon, 1],
		\end{cases}
		\\
		\Sigma^b \cap \RR^3 \times \{s\} &=
		\begin{cases}
			\phi_{1 - \tfrac{s }{\epsilon}}(\Sigma_0);& s \in [0,\epsilon], \\
			\phi_{\tfrac{s }{\epsilon} - 1}(\Sigma_0);& s \in [\epsilon, 2 \epsilon], \\
			\Sigma'_{3s - 6 \epsilon};& s \in [2 \epsilon, 3 \epsilon], \\
			\Sigma'_s; & s \in [3 \epsilon, 1 - 3 \epsilon],\\
			\Sigma'_{3s -2 + 6\epsilon};& s \in [1- 3\epsilon, 1- 2 \epsilon ],\\
			\phi_{\tfrac{1-s}{\epsilon}-1}(\Sigma_1)& s \in [1 - 2 \epsilon, 1 - \epsilon], \\
			\phi_{1 - \tfrac{1-s}{\epsilon}}(\Sigma_1); & s \in [1 - \epsilon, 1],
		\end{cases}
		\\
		\Sigma^c \cap \RR^3 \times \{s\} &=
		\begin{cases}
			\Sigma_0';& s \in [0, 2\epsilon], \\
			\Sigma'_{3s - 6 \epsilon};& s \in [2 \epsilon, 3 \epsilon], \\
			\Sigma'_s; & s \in [3 \epsilon, 1 - 3 \epsilon],\\
			\Sigma'_{3s -2 + 6\epsilon};& s \in [1- 3\epsilon, 1- 2 \epsilon],\\
			\Sigma_1'; & s \in [1 - 2 \epsilon, 1].
		\end{cases}
	\end{align*}
	In general, these surfaces need to be smoothed in the intersections of the defining intervals, but this can be done locally and without changing the
	ambient isotopy class.
	Notice now that we have a commutative diagram of ambient isotopies
	\begin{figure}[H]
		\centering
		\begin{tikzcd}
			\Sigma \ar[r, "\alpha"] \ar[rrrr, "\phi", bend right = 20] & \Sigma^a \ar[r, "\beta"] & \Sigma^b \ar[r, "\gamma"] & \Sigma^c \ar[r, "\delta"] & \Sigma',
		\end{tikzcd}
	\end{figure}
	\noindent where
	\begin{itemize}
		\item $\alpha$ is supported in $\RR^3 \times ([0, 3 \epsilon] \cup [1 - 3 \epsilon, 1])$,
		\item $\beta$ is supported in $\RR^3 \times [ \epsilon, 1 - \epsilon]$,
		\item $\gamma$ is supported in $\RR^3 \times ([0, 2 \epsilon] \cup [1 - 2 \epsilon, 1])$ and
		\item $\delta$ is supported in $\RR^3 \times ([0, 3 \epsilon] \cup [1 - 3 \epsilon, 1])$.
	\end{itemize}
	Here, $\alpha$, $\gamma$ and $\delta$ are obtained from lemma~3.3.5 in~\cite{K}, as in all three cases, the
	two relevant surfaces are both ambient isotopic to a third surface that is cylindrical in a collar neighbourhood
	of the boundary (note that the ambient isotopies provided by this lemma are supported in the collar neighbourhoods listed above), and $\beta$ is
	obtained from $\phi$ in the obvious manner.
\end{proof}

We can now devise an algorithm for computing the fundamental quandle.

\begin{theorem}
	\label{thm:algo}
	A presentation of the fundamental quandle
	$Q(\RR^3 \times I, \Sigma)$ is obtained from the motion picture of $\Sigma$ by the following algorithm:
	\begin{enumerate}
		\item Start with a presentation of $Q(\RR^3, L_0)$, obtained as in example~\ref{example:classic_quandle}.
		\item {
			Follow the stills along the motion picture.
			\begin{itemize}
				\item For each critical point of index 0 (move (i) on figure~\ref{fig:ESI}), add a new generator.
				\item {
					For each critical point of index 1 (move (ii) on figure~\ref{fig:ESI}), add a new relation 
					identifying the two elements represented by the arcs that the 1-handle attaches to
					(i.e. the two arcs in the stills before the critical value).}
			\end{itemize}
			Continue until reaching the end of the motion picture. The final presentation is a presentation of $Q(\RR^3 \times I, \Sigma)$.
		}
	\end{enumerate}
\end{theorem}
\begin{proof}
	Let $0 < t_1< \dots< t_n < 1$ be the critical values corresponding 
	to the critical points $p_1, \dots, p_n$ of $\pi|_{\Sigma}$, respectively.
	Choose numbers
	$$0 = b_0 < a_1 < t_1 < b_1 < a_2 < t_2 < b_2 < \dots < a_n < t_n < b_n < a_{n+1} = 1$$
	and $\epsilon > 0$ such that
	$$\epsilon < \min\Bigl(\min_{i=1, \dots, n}(t_i - a_i), \min_{i=1, \dots, n}(b_i - t_i), \min_{i=1, \dots, n+1}(a_i - b_{i-1})\Bigr).$$
	Up to ambient isotopy, we may assume that for each $i = 1, \dots, n$ the parts 
	$\Sigma _ {(a_i - \epsilon, b_i + \epsilon)}$
	are cylindrical (a product of some 1-manifold and an interval) everywhere except
	in small neighbourhoods of the critical point $p_i$.
	Define the following notations for the fundamental quandles: 
	\begin{align*}
		&Q(t) = Q(\RR^3 \times [0, t + \epsilon), \Sigma_{[0, t + \epsilon)}) &&\forall t \in [0,1- \epsilon],\\
		&Q^p_i = Q(\RR^3 \times (a_i - \epsilon, b_i + \epsilon), \Sigma_{(a_i - \epsilon, b_i + \epsilon)}) &&\forall i \in \{1, \dots, n\},\\
		&Q^r_i = Q(\RR^3 \times (b_i - \epsilon, a_{i+1} + \epsilon), \Sigma_{(b_i - \epsilon, a_{i+1} + \epsilon)})&&\forall i \in \{1, \dots, n-1\},\\
		&Q^a_i = Q(\RR^3 \times (a_i - \epsilon, a_i + \epsilon), \Sigma_{(a_i - \epsilon, a_i + \epsilon)}) &&\forall i \in \{1, \dots, n\} \quad \text{and}\\
		&Q^b_i = Q(\RR^3 \times (b_i - \epsilon, b_i + \epsilon), \Sigma_{(b_i - \epsilon, b_i + \epsilon)})&&\forall i \in \{1, \dots, n\}.
	\end{align*}
	We may now compute the fundamental quandle $Q(\RR^3 \times I, \Sigma)$ inductively:
	\begin{itemize}
		\item {
			By remark~\ref{remark1}, $Q(a_1) \cong Q(\RR^3 \times [0, a_1], \Sigma_{[0,a_1]})$.
			As $[0, a_1]$ contains no critical points of $\pi|_{\Sigma}$, 
			we have by lemma~\ref{lemma2} that $Q(a_1) \cong Q(\RR^3, L_0)$.
		}
		\item {
			To obtain a presentation of $Q(b_i)$ from a presentation of $Q(a_i)$, we use the Seifert-Van-Kampen theorem for quandles (theorem~\ref{thm:svk}):
			$$Q(b_i) \cong Q(a_i) *_{Q^a_i} Q^p_i.$$
			By lemma~\ref{lemma6}, the quandle homomorphism $Q^a_i \to Q^p_i$ induced by the inclusion 
			$\Sigma_{(a_i - \epsilon, a_i + \epsilon)} \hookrightarrow \Sigma_{(a_i - \epsilon, b_i + \epsilon)}$ 
			is the same as the one induced by $L_{a_i} \hookrightarrow \Sigma_{[a_i, b_i ]}$. Since $\Sigma_{[a_i, b_i ]}$
			has only one critical point, this homomorphism falls under one of the following three categories:
			\begin{itemize}
				
				\item {
					The critical point is of index $0$. Then, by lemma~\ref{lemma3},
					we have that $Q(b_i) \cong Q(a_i) * \langle x \rangle$, where $x$ is the generator of the quandle of the newborn component.
				}
				\item { 
					The critical point is of index $1$. Then, by lemma~\ref{lemma4},
					we have that $Q(b_i) \cong \tfrac{Q(a_i)}{r}$, where $r$ is the relation identifying the two elements represented by the arcs that the 1-handle attaches to.
				}
				\item {
					The critical point is of index $2$. Then, by lemma~\ref{lemma5},
					we have that $Q(b_i) \cong Q(a_i)$.
				}
			\end{itemize}
		}
		\item {
			To obtain a presentation of $Q(a_{i+1})$ from a presentation of $Q(b_i)$, we again use the Seifert-Van-Kampen theorem for quandles:
			$$Q(a_{i+1}) \cong Q(b_i) *_{Q^b_i} Q^{r}_{i}.$$
			Like before, by lemma~\ref{lemma6}, the quandle homomorphism $Q^b_i \to Q^{r}_{i}$ induced by the inclusion
			$\Sigma_{(b_i - \epsilon, b_i + \epsilon)} \hookrightarrow \Sigma_{(b_i - \epsilon, a_{i+1} + \epsilon)}$
			is the same as the one induced by $L_{b_i} \hookrightarrow \Sigma_{[b_i, a_{i+1}]}$.
			Since $\Sigma_{[b_i, a_{i+1}]}$ has no critical points, this homomorphism is an isomorphism
			by lemma~\ref{lemma2}.
			
		}
		\item {
			Finally, notice that by remark~\ref{remark1},
			the quandle homomorphism $Q^b_n \to Q(\RR^3 \times I, \Sigma_{(b_n - \epsilon, 1]})$ induced by the inclusion
			$\Sigma_{(b_n- \epsilon, b_n + \epsilon)} \hookrightarrow \Sigma_{(b_n - \epsilon, 1]}$ 
			is the same as the one induced by the inclusion $L_{b_n} \hookrightarrow \Sigma_{[b_n, 1]}$, 
			which is an isomorphism by lemma~\ref{lemma2}, as $\Sigma_{[b_n, 1]}$ contains no critical points.
			Then again by the  Sei\-fert-Van-Kam\-pen theorem
			$$ Q(\RR^3 \times I, \Sigma) \cong Q(b_n) *_{Q^b_n} Q(\RR^3 \times (b_n - \epsilon, 1], \Sigma_{(b_n - \epsilon, 1]}) \cong Q(b_n).$$
		}
	\end{itemize}

	We provide an additional argument as to why the algorithm is well-defined, that is, that the motion pictures of ambient isotopic surfaces
	give isomorphic quandles. By lemma~\ref{lemma7}, it is enough to show that the algorithm is well-defined up to 
	ambient isotopies supported in the interior of $\RR^3 \times I$ and ambient isotopies supported in the collar
	neighbourhood of $\RR^3 \times \{0,1\}$ which fix all the critical points.
	
	It was shown by Carter and Saito in~\cite{CS2}
	that the former can be diagrammatically represented by a set of 15 moves, called movie moves, so it is enough
	to check that two motion pictures related by the movie moves give isomorphic quandles. In fact, it is enough to check
	only the five moves that contain a $0$-handle or a $1$-handle, as the others don't change the quandle presentation.
	These simple calculations are carried out at the end of the article in figure~\ref{fig:movie_moves} (where the arcs in the frames are oriented in any compatible way).
	The remaining ambient isotopies fix all critical points, and therefore don't change the quandle presentation.

\end{proof}

\begin{remark}
	Note that when working with a presentation of a motion picture consisting of only finitely many diagrams of stills differing consecutively by planar isotopy and  moves
	(i) through (vi) on figure~\ref{fig:ESI}, only moves (i) and (ii) change the presentation, while the moves
	(iii) through (vi) may create or delete arcs in the diagrams, potentially changing which
	elements of the quandle are represented by arcs in the stills. One must keep track of these changes in order to be able to properly form the relations induced by the
	critical points of index 1.
\end{remark}

\begin{example}
	Using the algorithm from theorem~\ref{thm:algo}, we calculate the fundamental quandle of the surface depicted by
	the motion picture on figure~\ref{fig:movie}. Note that between the second and the third still, four Reidemeister 2 moves were performed.
	In the first still, we start with the presentation of the trefoil 
	$$\langle a, b \mid (a \tr b) \tr a =b, b \tr (a \tr b)=a \rangle.$$
	Between the first and the second frame, we add a $0$-handle, resulting in a new generator $x$ and changing our presentation to
	$$\langle a, b, x \mid (a \tr b) \tr a =b, b \tr (a \tr b)=a \rangle.$$
	Between the second and the third frame, no new handles were added. The presentation therefore stays the same, while some new arcs, representing certain elements of the presented quandle (like $(x \tl a) \tl x$), have appeared. 
	Finally, between the third and the fourth frame, a $1$-handle was added, resulting in a new relation identifying $a \tr b$ and $(x \tl a) \tl x$. The fundamental quandle of the surface
	then has a presentation
	$$\langle a, b, x \mid (a \tr b) \tr a =b, b \tr (a \tr b)=a, a \tr b = (x \tl a) \tl x \rangle.$$
\end{example}

\begin{figure}[H]
    \centering
	\includegraphics[width = \linewidth]{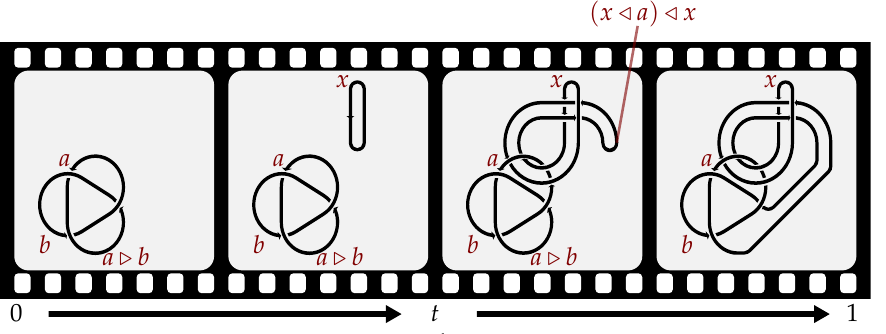}
    %\includesvg[width = \linewidth]{movie.svg}
    \caption{A sequence of stills for a surface.}
    \label{fig:movie}
\end{figure}

\section{The fundamental quandle of a ribbon concordance}
\label{sec4}
The study of knot concordance represents a bridge between classical knot theory and surface-knot theory.

\begin{definition}
	A \textit{concordance} between two knots $K_{0}$ and $K_{1}$ is the image $C \subset \s^3 \times I$ of a (smooth) proper embedding 
	$f\colon (\s^{1}\times I, \s^1 \times \{0\}, \s^1 \times \{1\}) \to (\s^3 \times I, K_0 \times \{0\}, K_1 \times \{1\})$.
	We may assume that restriction $\pi |_C$ of the projection $\pi:\s ^{3}\times I\to I$ is a Morse function, 
	and if it has no critical points of index $2$, $C$ is called a \textit{ribbon concordance} from $K_1$ to $K_0$. 
	A knot $K$ is called \textit{ribbon} if there exists a ribbon concordance from $K$ to the unknot.
\end{definition}

It is clear that concordance is an equivalence relation of knots.
Ribbon concordance, however, is not symmetric. 
In his seminal paper~\cite{G}, Gordon introduced notation $K_{1}\geq K_0$ if there exists a ribbon concordance from $K_1$ to $K_0$,
and conjectured that the relation $\geq $ gives a partial order on the set of all knots. Over the following decades,
many invariants were shown to behave monotonically with respect to ribbon concordance, giving credence to this conjecture, until
it was finally confirmed forty years later by Agol~\cite{A}.
Many other questions about ribbon concordance, such as if it is well-founded, are, as of writing, still open.

We now restate and prove the main theorem of this article.

\begin{thm_again}{thm:ribbon_conc_quandle}{}
	If $C \subset \s^3\times I$ is a ribbon concordance from $K_1$ to $K_0$, 
	then the inclusions of pairs $i_0: (\s^3 \times \{0\}, K_0 \times \{0\}) \hookrightarrow (\s^3 \times I, C)$
	and $i_1:(\s^3 \times \{1\}, K_1 \times \{1\}) \hookrightarrow (\s^3 \times I, C)$ induce a quandle monomorphism
	$\iota_0: Q(\s^3, K_0) \hookrightarrow Q(\s^3 \times I, C)$ 
	and a quandle epimorphism $\iota_1: Q(\s^3, K_1)\twoheadrightarrow Q(\s^3 \times I,C)$,
	respectively.
\end{thm_again}

\begin{remark}
	In section~\ref{sec3}, we obtained some results about surfaces in $\RR^3 \times I$ (namely theorem~\ref{thm:algo}). 
	In this section, we deal with surfaces in $\s^3 \times I$ instead. 
	Note that the results from the previous section all directly translate to this setting.
\end{remark}

\begin{proof}
	As $C$ is an annulus, we have $\chi(C) = 0$, and as $\pi_C$ has no critical points of index $2$, it must have the same number of
	critical points of index $0$ and critical points of index $1$. Let $k$ be that number.
	We may assume that, up to ambient isotopy, that the $k$ critical points of index $0$
	are mapped to $\tfrac{1}{2k+2}$, while the $k$ critical points of index $1$ are mapped to $\tfrac{1}{k+1}, \dots, \tfrac{k}{k+1}$ 
	in such order that each consecutive 
	$1$-handle connects some disk to the main body of $C$ (the component with boundary at height $0$). To prove that $\iota_0$ is a monomorphism, we show that
	given two paths representing elements of $Q(\s^3, K_0)$ and a homotopy in $\s^3 \times I - N_C$ between them, one can gradually transform
	this homotopy to a homotopy between the same two paths lying in $\s^3 \times \{0\}$.

	We begin by establishing some notation. Let $N_C \subset \s^3 \times I$ and $N_{K_0} \subset \s^3 \times \{0\}$
	be the normal disk neighbourhoods of $C$ and $K_0$, respectively, and define their exteriors
	$E_C$ and $E_{K_0}$. Pick a basepoint $z \in E_{K_0}$.
	Define $$C_i := C \cap \s^3 \times [0, \tfrac{i}{k+1}]$$
	for all $i \in \{1, \dots, k\}$. We can deform a neighbourhood of the critical point of index $1$ at height $\tfrac{i }{k+1}$
	to a critical band $B_i \subset \s^3 \times \{\tfrac{i }{k+1}\}$ 
	attached on one side to a disk $D_i \subset \s^3 \times [0, \tfrac{i }{k+1}]$ with exactly one critical point (of index $0$)
	and on the other to the main body of $C_i$, that is the component
	of $$\Sigma_i := C_i - B_i - D_i$$ with nonempty boundary at height $0$. 
	Define the normal disk bundle neighbourhoods
	$N_{\Sigma_i}, N_{C_i} \subset \s^3 \times [0,\tfrac{i}{k+1}]$ of $\Sigma_i$ and $C_i$, respectively, such that
	$$N_{K_0} \subset N_{\Sigma_1} \subset N_{C_1} \subset \dots
	N_{\Sigma_i} \subset N_{C_i} \subset N_{\Sigma_{i+1}} \subset \dots \subset N_{\Sigma_k} \subset N_{C_k} \subset N_C,$$
	along with their respective exteriors $E_{\Sigma_i}, E_{C_i} \subset \s^3 \times [0,\tfrac{i}{k+1}]$.
	Lastly, since multiple copies of the interval $I$ with distinct meanings will appear, 
	we index some of them by the variable we use; $t$ for $I_t$, $s$ for $I_s$ and $r$ for $I_r$.

	Let now $\alpha$ and $\beta$ be two paths representing elements in $Q(\s^3, K_0)$, that is two paths 
	$$ \alpha, \beta: (I_t, \{0\}, \{1\}) \to (E_{K_0}, \partial N_{K_0}, \{z\}).$$
	Then these paths also represent elements in 
	$Q(\s^3 \times I, C)$. If $\alpha$ and $\beta$ represent the same element in $Q(\s^3 \times I, C)$,
	then there exists a homotopy 
	$$H\colon I_t \times I_s \to E_{C}$$
	such that $H(t,0) = \alpha(t)$, $H(t,1) = \beta(t)$,
	$H(0,s) \in \partial N_{C}$ and $H(1,s) = z$ for all $t \in I_t$ and $s \in I_s$. Our goal is to gradually transform this homotopy to one in $E_{K_0}$.

	Fixing $t=0$, $H(\{0\}_t \times I_s)$ is a path in $\partial N_{C}$ from $\alpha(0)$ to $\beta(0)$. Since
	$C \cap \s^3 \times (\tfrac{k }{k+1}, 1]$ contains no critical points of $\pi|_C$, it is ambient isotopic to the product of some knot with
	the interval $(\tfrac{k }{k+1}, 1]$. Isotoping $C$ to this form, we can then push the path $H(\{0\}_t \times I_s)$ 
	down along $N_C$ to always lie at or below height 
	$\tfrac{k }{k+1}$, meaning we
	push it to a path in $\partial N_{C_k}$. Next, we can regard both the disk $D_k$ and the band $B_k$ 
	as the product $I_{D_k} \times I_{D_k}$ and $I_{B_k} \times I_{B_k}$, respectively, such that $B_k$ is attached
	to $\Sigma_k$ along $I_{B_k} \times \{0\}_{B_k}$, and is attached to 
	$I_{D_k} \times \{0\}_{D_k} \subset D_k$ along $I_{B_k} \times \{1\}_{B_k}$. 
	The natural deformation retraction
	$I \times I \retract I \times \{0\}$, first used on $D_k$ and then on $B_k$, gives rise to a
	map that pushes $H(\{0\}_t \times I_s)$ along $\partial N_{C_k}$ to a path
	lying entirely in $\partial N_{\Sigma_k}$. Then, as $\Sigma_k \cap \s^3 \times (\tfrac{k-1}{k+1}, \tfrac{k }{k+1}]$
	contains no critical points of $\pi|_C$, we again isotope it to a cylindrical form and push $H(0,s)$ to a path in $\partial N_{C_{k-1}}$.
	Repeating this process, we can push $H(\{0\}_t \times I_s)$ to a path in $\partial N_{K_0}$.

	We have described a homotopy 
	$$\widetilde{H}_0\colon \{0\}_t  \times I_s \times I_r \to \partial N_{C} \subset E_C$$
	where $\widetilde{H}_0(0, s, 0) =  H(0,s)$ and $\widetilde{H}_0(\{0\}_t \times I_s \times \{1\}_r)$ is some path in $\partial N_{K_0}$ from $\alpha(0)$ to $\beta(0)$.
	We extend this map to a map on 
	$\{0,1\}_t \times I_s \times I_r \cup I_t \times \{0,1\}_s \times I_r$
	by defining $\widetilde{H}_0(1,s,r) = z$,
	$\widetilde{H}_0(t,0,r) = \alpha(t)$ and $\widetilde{H}_0(t,1,r) = \beta(t)$ for arbitrary $t$, $s$ and $r$.
	Since $(I_t \times I_s, \{0,1\}_t \times I_s \cup I_t \times \{0,1\}_s )$ is a CW-pair, it has the homotopy extension property. Thus, taking 
	$H$ as the map on $I_t \times I_s$ and taking $\widetilde{H}_0$ as the initial homotopy, there exists
	a homotopy
	$$\widetilde{H}: I_t \times I_s \times I_r \to E_{C}$$
	such that
	$\widetilde{H}(0,s,r) = \widetilde{H}_0(0, s, r) \in  \partial N_{C}$, 
	$\widetilde{H}(t,0,r) = \alpha(t)$, $\widetilde{H}(t,1,r) = \beta(t)$,
	$\widetilde{H}(1,s,r)= z$ and $\widetilde{H}(t,s,0) = H(t,s)$ for all $t \in I_t, s \in I_s$ and $r \in I_r$.
	Then $\widetilde{H}(I_t \times I_s \times \{1\}_r)$ is a homotopy between $\alpha$ and $\beta$ such that 
	$\widetilde{H}(\{0\}_t \times I_s \times \{1\}_r) \subset \partial N_{K_0}$ and $\widetilde{H}(\{1\}_t \times I_s \times \{1\}_r) \subset \{z\}$.

	The union of paths $\alpha$, $\beta$ and $\widetilde{H}(\{0\}_t \times I_s \times \{1\}_r)$ is a loop $\gamma$ in $E_{K_0} \subset E_C$
	and $\widetilde{H}(I_t \times I_s \times \{1\}_r)$ is its disk filling in $E_C$, meaning that $\gamma$ represents a trivial element in $\pi_1(E_C)$.
	By Gordon's lemma~3.1 in~\cite{G} (lemma~\ref{lemma:gordon} in this paper), the inclusion-induced map $\pi_1(E_{K_0}) \to \pi_1(E_C)$ is injective
	(the lemma actually states this result for the inclusion $\s^3 - K_0 \hookrightarrow \s^3  \times I - C$, but since $E_{K_0} $ and $E_C$ are
	strong deformation retracts of $\s^3 - K_0$ and $\s^3 \times I -C$, respectively, it also holds for the inclusion $E_{K_0} \hookrightarrow E_C$).
	Then there must also exist
	a disk filling of $\gamma$ in $E_{K_0}$, which, when properly parametrized, gives us the desired homotopy between $\alpha$ and $\beta$
	lying entirely in $E_{K_0}$. This completes the proof that $\iota_0$ is a monomorphism.

	What remains is to prove that $\iota_1$ is an epimorphism.
	To do so, we can look at
	the dual motion picture; that is the motion picture of $C$, from the top to the bottom, consisting of $k$ $1$-handles and $k$ $2$-handles. To obtain a presentation of
	$Q(\s^3 \times I, C)$, we use the algorithm from theorem~\ref{thm:algo}: we start with a presentation of $Q(\s^3, K_1)$ then add $k$ relations when passing over the $1$-handles
	and then do nothing when passing over the $2$-handles. Thus, the inclusion $(\s^3 \times \{1\}, K_1 \times \{1\}) \hookrightarrow (\s^3 \times I, C)$ induces an
	epimorphism, completing the proof.
\end{proof}

Theorem~\ref{thm:ribbon_conc_quandle} provides, in theory, powerful obstructions to ribbon concordances between knots based on their fundamental quandles.
In practice, useful (minimal) presentations of fundamental quandles are often quite difficult to obtain and work with.
There are some families of knots where such presentations can be obtained quite easily: for example knots that can be represented
as closures of braids with two strands, which includes $T(p,2)$ torus knots, will have fundamental quandles of rank two (with the sole exception of the unknot, which has
a trivial fundamental quandle).
More generally, any knot that can be represented as a closure of a braid with $q$ strands will have a fundamental quandle of rank $\leq q$
(this includes torus knots $T(p,q)$ for coprime integers $1 < q< p$, for which this bound is sharp).
It is, however, difficult to show that quandle presentations obtained from such diagrams are minimal in terms of number of generators.

It would be interesting to see how invariants defined from the fundamental quandle of a knot, such as colouring invariants and
quandle cocycle invariants, interact with ribbon concordances. They could possibly provide more computable obstructions to ribbon
concordances between knots. Such is the case with surface knots, where one can directly relate the cocycle invariants of 
ribbon concordant surface-knots, as was shown in~\cite{CKS3}. The proof used there, however, does not
directly translate to the theory of ribbon concordances between classical knots; for a ribbon concordance from $K_1$ to $K_0$, a quandle colouring
of $K_1$ does not necessarily induce a quandle colouring of $K_0$. As such, a more subtle approach would be needed. So far, the authors haven't had
any luck in finding better obstructions.

\newpage
\begin{figure}[H]
	\includegraphics[width=14cm]{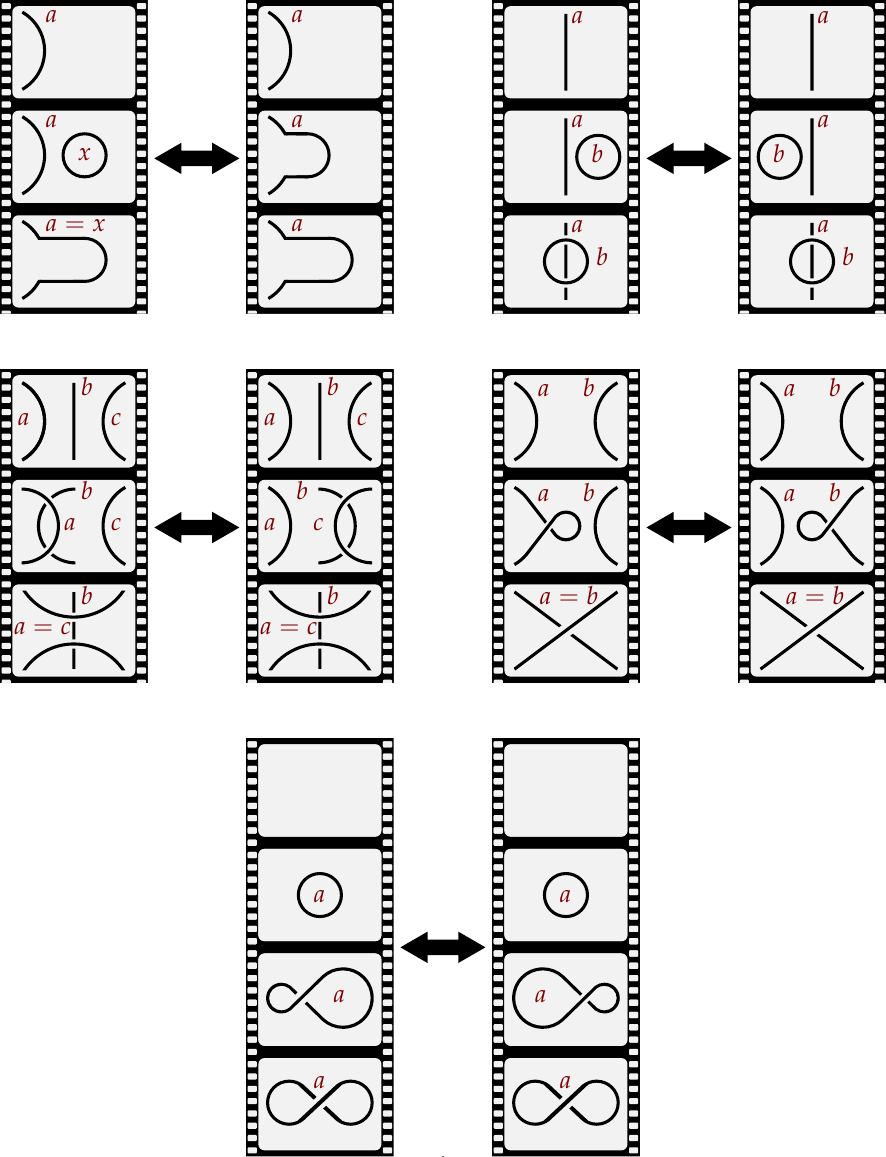}
	\caption{The effect of movie moves on the fundamental quandle presentation.}
	\label{fig:movie_moves}
\end{figure}
\newpage

\printbibliography

\end{document}